\documentclass[11pt,fleqn]{article}
\usepackage{amsmath,amsfonts,amsthm,amssymb}
\newcommand{\D}{\mathbb D}
\newcommand{\C}{\mathbb C}
\newcommand{\R}{\mathbb R}
\newcommand{\N}{\mathbb N}
\newcommand{\Z}{\mathbb Z}
\newcommand{\T}{\mathbb T}
\newcommand{\G}{\mathcal G}

\newcommand{\e}{\mathbf e}
\newcommand{\uu}{\mathbf u}

\newcommand{\tv}{\mathbf t}
\newcommand{\EN}{\mathcal E_N}

\newtheorem{theorem}{Theorem}
\newtheorem{lemma}[theorem]{Lemma}
\newtheorem{corollary}[theorem]{Corollary}
\renewcommand{\Re}{\mathop {\rm Re}\nolimits}
\renewcommand{\Im}{\mathop {\rm Im}\nolimits}

\title{Inner functions and Pick-Nevanlinna interpolation in a multi-connected domain. }

\author{Michel Crouzeix, \footnote{Univ.\,Rennes, CNRS, IRMAR\,-\,UMR\,6625, F-35000 Rennes, France, email: michel.crouzeix@univ-rennes.fr}}

\begin{document}
\maketitle
\begin{abstract}
We describe the set of inner functions of finite order in a multi-connected domain, then we consider an optimization formulation of the Pick-Nevanlinna interpolation problem, and we generalize it to Hermite type interpolation.
\end{abstract}

\paragraph{2000 Mathematical subject classifications\,:} 30D50; 30D99

\noindent{\bf Keywords\,:}{Inner functions, Pick-Nevanlinna interpolation}

\section{Introduction}
We consider a bounded connected domain of the complex plane $\Omega\subset \C$ with $k$ holes.
We denote by $\Gamma _0$ its exterior boundary and by  $\Gamma _j$, $j=1,\dots,k$, 
the interior ones\,; these curves are assumed analytic and pairwise
distinct\,: $\Gamma _i\cap\Gamma _j=\emptyset$ if $i\neq j$. 
Let  $r$ distinct points $\zeta_1, \zeta _2,\dots,\zeta _r$ in $\Omega$ 
and $r$ complex numbers $\tau _1, \tau_2,\dots,\tau_r,$ not all zero, be given.
The classical Pick-Nevanlinna problem concerns the existence of analytic functions $f$ on $\Omega$
with modulus bounded by 1 such that $f(\zeta _j)=\tau _j$, $1\leqslant j\leqslant r$.  
In this paper, we look at a slightly different problem,
\begin{align*}
(P)\quad\left\{\begin{matrix}
\textrm{Find a holomorphic function } f^* \textrm{ which realizes} ,\hfill\\[10pt]
m^*=\min\{\|f\|_\infty\,: f\ \in H^\infty(\Omega),\  f(\zeta _1)=\tau _1,\dots,f(\zeta _r)=\tau _r\}.
\end{matrix}
\right.
\end{align*}
Here $H^\infty(\Omega)$ is the space of bounded holomorphic functions on $\Omega$ equipped with the norm $\|f\|_{\infty}=\sup_{z\in\Omega}|f(z)|$.
Note that such a function $f^*$ realizing this minimum exists. Indeed, the set of functions satisfying the 
constraints is nonempty; for instance, the Lagrange interpolation polynomial 
may be chosen. Then, existence follows from the compactness of the restriction map $f\mapsto f_{|K}$ for any compact set $K\subset\Omega$.

An aim of this paper is to show that the function $f^*$ which realizes $m^*$ is unique and that $f^*/m^*$ is an inner function in $\Omega$ 
of order less or equal to $r{+}k{-}1$. 
Our approach is slightly different, but the key ideas are due to Garabedian and Ahlfors and may be found in the article \cite{gara}, but where 
some proofs need to be completed. In fact, I have written this manuscript since I would have a generalization of the Pick-Nevanlinna theory  
to the Hermite type interpolation. I needed this for proving the following result
\begin{theorem}\label{th1}
 Let $A\in \C^{d,d}$ be a $d\times d$ matrix with eigenvalues in $\Omega$. Then, there exists  an inner function $f_0$ of order 
 $\leqslant d{+}k{-}1$ which realizes
 \[
 \|f_0(A)\|=\max\{\|f(A)\|\,: f\in H^\infty(\Omega), |f|)\leqslant1\textrm{ in  }\Omega\}.
 \]
\end{theorem}

The organization of this note is the following. In Section\,\ref{S2} we define and study the inner functions of finite order in 
$\Omega$ (which have some similarity to the Blaschke products). In Section\,\ref{sch} we consider a problem of Schwarz type. 
In Section\,\ref{pn} we analyze an optimization  problem
similar to $(P)$ but restricted to inner functions of finite order.
In Section\,\ref{picnev} we show that this inner function multiplied by $m^*$ is the unique solution of $(P)$. 
In Section\,\ref{picder} we generalize problem $(P)$ by introducing interpolation of derivatives
and we prove Theorem\,\ref{th1}.
In Section\,\ref{link}, we show that the original Pick-Nevanlinna problem has a solution if and only if 
$m^*=m^*(\zeta ,\tau )\leqslant1$, and we look at the dependence of $m^*$ with respect to the data $\zeta _j$ and $\tau _j$. 
Finally, in the Appendix, we prove a result concerning the harmonic measures of the
boundary curves. 
\medskip

We shall say that a function $f$ is analytic in $\overline\Omega$, if $f$ has an analytic extension 
in an open neighborhood of $\overline\Omega$. \medskip

\noindent{\bf Remark. }{\it In order to avoid some technical difficulties, we have assumed  that
the boundary  curves are analytic but the use of a conformal transformation allows to extend the results to Jordan curves.}
\smallskip

\section{Inner functions of finite order}\label{S2}

We denote by $\D=\{z\in \C\,: |z|<1\}$ the open unit disk and  by $\T=\{z\in \C\,: |z|=1\}$ its boundary.
\medskip

\noindent{\bf Definition.} {\it A function $f$ is an inner function of order $n$ in $\Omega$ if $f$ is holomorphic in $\Omega$, if $\|f\|_\infty=1$ 
and if for all $\tau  \in \D$ the equation $f(z)=\tau $ has exactly $n$ solutions in $\Omega$ (counting their multiplicity). In particular, 
a constant function of modulus $1$\,: $f(z)=\lambda\in \T $, is an inner function of order $0$. 
}\medskip

If $f$ is an analytic function in a neighborhood of $\Omega$ such that $|f|=1$ on the boundary $\partial \Omega$, then, from  
the argument principle, the number $n$ of zeros of $f(z){-}\tau$ is finite and independent of $\tau \in \D$, whence $f$ 
is an inner function of order $n$ on $\Omega$. \medskip

We write $z=x{+}iy$ the generic point in $\C$\,; we denote by $d\sigma$ the arclength on the boundary 
of $\Omega$, $\vec\nu=(\nu_1,\nu_2)$ is the unit outward normal, $\partial _\nu=\nu_1\partial _x{+}\nu_2\partial_y$ 
is the normal derivative.
We use the harmonic measures $\omega_i$ of the boundary curves $\Gamma _i$ defined by \cite{ahlf}
\[
-\Delta \omega_i=0\quad\textrm{ in }\Omega, \quad\omega_i=\delta_{ij}\ \textrm{( Kronecker symbol)  on  }\Gamma_j,\ 
j=0,\dots,k.
\]
Clearly, these functions are linearly independent; thus the $k\times k$ symmetric matrix $M$, 
with entries $m_{ij}=\int_\Omega\nabla \omega_i.\nabla \omega_j\,dxdy$, $i$ and $j\in \{1,\dots,k\}$,
is positive definite. Recall also that, by Green formula,
\[
m_{ij}=\int_\Omega\nabla \omega_i.\nabla \omega_j\,dxdy=\int_{\Gamma _i}\partial _\nu\omega_j\,d\sigma .
\]
We denote by $N=(n_{ij})=M^{-1}$ the inverse matrix of $M$.

If $\varphi $ is holomorphic in $\Omega$, nonvanishing in a neighborhood of $\Gamma _j$, 
and if $\gamma _j$ is a  closed curve contained in this neighborhood homotopic to $\Gamma _j$, the integer
\[
\rho _j(\varphi )=\frac{1}{2\pi i}\int_{\gamma_j}\frac{\varphi '(z)}{\varphi (z)}\,dz
\] 
does not depend on the choice of $\gamma_j$ and represents the number of rotations of $\varphi $ on $\Gamma _j$. 
Furthermore, if $\varphi '$ has a continuous extension on $\Gamma _j$ and if $\varphi $ does not vanish on $\Gamma _j$, then it holds 
$\rho _j(\varphi )=\frac{1}{2\pi i}\int_{\Gamma_j}\frac{\varphi '(z)}{\varphi (z)}\,dz=\frac{1}{2\pi }\int_{\Gamma _j}\partial _\nu\log|\varphi |\,d\sigma 
$.

Let us consider a point $p\in\Omega$\,; we denote by  $G_p$ the Green function in this point. Recall that $G_p$ vanishes 
on the boundary $\partial \Omega$, and 
\[
\int_\Omega\nabla G_p.\nabla v\,dxdy=v(p), \quad\textrm{for all smooth functions $v$ vanishing on }\partial \Omega.
\]
Hence $-\Delta G_p=\frac{1}{2\pi} \,\Delta \log|\cdot-p|=\delta (\cdot-p)$, in the distribution sense in $\Omega$.
Also, for $z\in\Omega$, $G_p(z)$ is a harmonic function in $p\neq z$. Note that 
$\int_\Omega\nabla G_p.\nabla \omega_i\,dxdy=0$ since $\omega_i$ is harmonic and $G_p$ vanishes on the boundary;
therefore
\[
0=\int_\Omega\nabla G_p.\nabla \omega_i\,dxdy=\int_{\partial \Omega}\partial _\nu G_p\ \omega_i \ d\sigma +\omega_i(p),
\]
which shows that $\int_{\Gamma _i}\partial _\nu G_p\ d\sigma=-\omega_i(p) $.\medskip

\noindent{\bf Remark. }{\it The functions $\omega_i$ and  $G_p$ have harmonic extensions in a neighborhood of $\Omega$, 
since the boundary has been assumed analytic\,; in particular, they are $C^\infty $ functions on the boundary.}

\begin{lemma} \label{lem1} For each $p\in\Omega$, 
there exists one and only one holomorphic function $\varphi _p$ in $\Omega$ such that
\begin{itemize}
\item $p$ is the unique root of $\varphi_p $ in $\Omega$ and it holds $\varphi_p' (p)>0$,
\item $|\varphi_p |=1$ on $\Gamma _0$,
\item $|\varphi_p |$ is constant on $\Gamma _j$ and $\rho _j(\varphi_p )=0$, for $j=1,\dots,k$.
\end{itemize}
Furthermore, the function $\varphi _p$  is holomorphic in $\overline\Omega$. 
\end{lemma}
\begin{proof}
{\it Existence.}\\
 We set $h=2\pi \sum_{j=1}^kc_j(p)\omega_j-2\pi G_p-\log|\cdot-p|$, with $
c_i(p)=-\sum_{j=1}^kn_{ij}\omega_j(p)$, $i=1,\dots,k$\,: thus $\Delta h=0$ in $\Omega$. Note that, if $i\geqslant1$
\[
\int_{\Gamma _i}\partial _\nu(\sum_{j=1}^kc_j(p)\omega_j-G_p)\,d\sigma =\sum_{j=1}^km_{ij}c_j(p)+\omega_i(p)=0,
\]
it also holds 
$\int_{\Gamma _i}\partial _\nu\log|\cdot-p| \,d\sigma =0$ since $p$ is exterior to $\Gamma _i$, whence  $\int_{\Gamma _i}\partial _\nu h \,d\sigma =0$ if $i\geqslant 1$.
Furthermore, $0=\int_\Omega \Delta h\, dxdy=\sum_{i=0}^k\int_{\Gamma _i}\partial _\nu h\,d\sigma$, which shows that $\int_{\Gamma _0}\partial _\nu h\,d\sigma=0$.
Therefore,  the harmonic function $h$ has a conjugate harmonic function $g$ such that
$h{+}ig$ is holomorphic in $\Omega$ and that we fix by setting $g(p)=0$. Then, we set 
$\varphi _p(z)=(z{-}p)\exp(h(x,y){+}ig(x,y))$ with $z=x{+}iy$.

This function satisfies $\varphi _p(p)=0$, $\varphi'_p(p)=\exp(h(p))>0$, $\log|\varphi _p|=\log |\cdot -p|+h=2\pi \,c_i$ on $\Gamma _i$, if $i=1,\dots,p$, $\log|\varphi _p|=0$ on $\Gamma _0$. And finally, 
$\rho _j(\varphi_p)=\frac{1}{2\pi }\int_{\Gamma _j}\partial _\nu(\log|\cdot-p|+ h)\,d\sigma =0$ for $j\geqslant 1$.\medskip

\noindent{\it Uniqueness.}  Let $\varphi _p$ and $\tilde \varphi _p$ be
two functions satisfying the assumptions of the lemma\,; we set $\psi=\varphi _p/\tilde\varphi _p$ and $u=\log|\psi |$. 
The function $\psi$ is holomorphic in $\Omega$, $u=\log|\psi |$ is harmonic in $\Omega$, $u=0$ on $\Gamma _0$, $u$ 
has a constant value $c_i$ on $\Gamma _i$, $i=1,\dots,k$, whence $u=\sum_{j=1}^kc_j\,\omega_j$. From 
$0=\rho_i(\varphi _p)-\rho_i(\tilde\varphi _p)=\rho (\psi )=\frac{1}{2\pi }\int_{\Gamma _i}\partial _\nu u\,d\sigma =\sum_{j=1}^km_{ij}c_j$ 
for $i=1,\dots,k$, we deduce $c_j=0$; 
thus,  $u=0$ in $\Omega$, and then $\psi $ is constant with modulus 1 in $\Omega$\,;  with $\varphi _p=\psi \tilde\varphi _p$,  $\varphi _p'(p)>0$,  $\tilde\varphi _p'(p)>0$, we get $\psi =1$, which shows the uniqueness.
\end{proof}

The functions $\varphi _p$ are fully defined by
\begin{align}\label{eq1}
\left\{
\begin{matrix}&\varphi _p \textrm{ is holomorphic in }\Omega,\ \varphi '_p(p)>0, \rho_j(\varphi_p )=0\textrm{ for }j=1,\dots,k,\\
&\log|\varphi _p(z)|=-2\pi \Big(\sum_{i,j=1}^k n_{ij}\, \omega_i(p)\,\omega_j(z)+ G_p(z)\Big). \\
\end{matrix}
\right.
\end{align}
\begin{lemma}
 The functions $p\mapsto\log|\varphi _p(z)|$ and $p\mapsto\arg(\varphi _p(z)$ are harmonic in $\Omega$ with respect to the variable $p\neq z$.
\end{lemma}
\begin{proof}
 Note that $\log|\varphi _p(z)|=\log|\varphi _z(p)|$ since $G_p(z)=G_z(p)$; 
 thus,  $\log|\varphi _p(z)|$ is harmonic in $p\neq z$ since it is harmonic in $z\neq p$. This also implies that 
 $\arg(\varphi _p(z)$  is harmonic in $p\neq z$. Indeed, let $z_0\in\Omega$\,; for $z$ close to $z_0$ we can write
 \[
 \log|\varphi _p(z)|=2\sum_{j\geqslant 0}\Re (a_j(p)(z{-}z_0)^j)=\sum_{j\geqslant 0} (a_j(p)(z{-}z_0)^j{+}\overline{a_j(p)}\,(\bar z{-}\bar z_0)^j),
 \]
 which shows that the functions $a_j(p)$ are harmonic in $p$. It also holds
  \[
 i\arg(\varphi _p(z))=2i\sum_{j\geqslant 0}\Im (a_j(p)(z{-}z_0)^j)=\sum_{j\geqslant 0} (a_j(p)(z{-}z_0)^j{-}\overline{a_j(p)}\,(\bar z{-}\bar z_0)^j),
 \]
 which shows that $\arg(\varphi _p(z))$ is harmonic in $p$.
\end{proof}
\noindent{\bf Remark. }{\it  The functions $\log|\varphi _p(z)|$ and $\arg(\varphi _p(z))$
are harmonic conjugates regarded in variable $z$, but not with respect to variable $p$;
indeed, $\log|\varphi _p(z)|{+}i\arg(\varphi _p(z))$ is holomorphic in $z$ but, at least in general, not in $p$. Similarly, $\varphi _p(z)$ is holomorphic in $z$ but not in $p$. For instance, in the case of the unit disk $\Omega=\D$ and thus $k=0$, the corresponding 
function is $\varphi _p(z)=\frac{z-p}{1-\bar pz}$.}\medskip

Now, we fix a point $z_0\in \Gamma _0$ and consider the multiplicative group $\G$ of holomorphic functions $g$
in $\Omega$  which do not vanish in $\Omega$ and are such that
$|g|$ is constant on each $\Gamma _i$, $0\leqslant i\leqslant k$, $|g|=1$ on $\Gamma _0$ 
and $g(z_0)=1$.\medskip

\noindent{\bf Remark. }{\it From $\Delta \log|g| =0$ in $\Omega$ and $ \log|g|$ constant on each $\Gamma _i$ we deduce that 
$\log|g|$ has a harmonic extension close to $\overline\Omega$;
thus, $g$ is analytic in $\overline\Omega$ and $g(z_0)$ is well defined.}

\begin{lemma} 
 If $\rho =(\rho _1,\dots,\rho _k)^T\in\Z^k$, there exists a unique function $g_\rho \in \G$ such that $\rho _j(g_\rho )=\rho _j$, $j=1,\dots,k$, 
 and $g_\rho (z_0)=1$.
\end{lemma}
\begin{proof}{\it Existence.} We set $c_i=\sum_{j=1}^kn_{ij}\rho _j$ et $u=2\pi \sum_{j=1}^kc_{j}\omega_j$\,; then, 
\\$\int_{\Gamma _i}\partial _{\nu}u\,d\sigma = \sum_{j=1}^km_{ij}c_j=\rho _i\in\Z$, $i=1,\cdots,k$ and $\int_{\Gamma _0}\partial _{\nu}u\,d\sigma=-\sum_{i=1}^k\rho _i\in\Z$;  
thus, the harmonic function $u$ has a harmonic conjugate function $v$ such that $v(z_0)=0$. We set 
$g_\rho =\exp(u{+}iv)$; then, it holds $g_\rho \in\G$, $g_\rho (z_0)=1$ and $\rho _j(g_\rho )=\rho _j$.

The proof of uniqueness is similar to the proof in Lemma\,\ref{lem1}.
 \end{proof}

\noindent{\bf Remark.} {\it If $\rho =(1,0,\dots,0)^T$, then $g_\rho$ effects a one-to-one conformal mapping of $\Omega$ onto the annulus $\{1\!<\!|z|\!<\! r_1\}$ minus $k{-}1$ concentric arcs situated on the circles $\{|z|=r_i\}$, $i=2,\dots,k$, where $r_i$ denotes the constant value of $|\omega_\rho |$ on $\Gamma _i$.} 

\begin{lemma}
 The inner functions in $\Omega$ of order $n>0$ are the functions of the form
 \[
 f=\lambda \,g_\rho \,\varphi _{p_1}\cdots\varphi _{p_n},\quad \lambda \in\T,\ \rho =(\rho _1,\dots,\rho _k)^T\in\Z^k, 
 \]
 with $\rho _j=-\sum_{\ell=1}^n\omega_j(p_\ell)$ and where $p_1,\dots,p_n \in\Omega$ denote the roots of $f$ (repeating along the multiplicity). Furthermore, it necessary holds
 $n\geqslant k{+}1$.
 \end{lemma}
\begin{proof}
 a) If $f$ is of such a form,
$f$ is holomorphic in $\overline \Omega$; furthermore, $|f|=1$ on $\Gamma _0$. Also, on $\Gamma _i$ ($i\geqslant 1$), 
it holds $\log|g_\rho| =\sum_{j=1}^kn_{ij}\rho _j$,
 $\log|\varphi _{p_\ell}|=\sum_{j=1}^kn_{ij}\omega _j(p_\ell)$;  thus,  $\log|f|=0$ and then $|f|=1$. 
 This shows that $f$ is an inner function of order $r$.
 
 b)  Now, we observe that, if $f$ is inner of order $n$, then $|f|=1$ on the boundary. Indeed, otherwise it would exist 
 a sequence $z_l\in\Omega$ which tends to $z\in \partial \Omega$ such that $\zeta _l:=f(z_l)$ tends to $\zeta \in \D$. 
 But the inverse map defined by $f^{-1}(\tau )=\{z\in\Omega\,: f(z)=\tau \}$ (non ordered set with $n$ values) is continuous
 and $z_l\in f^{-1}(\zeta _l)$, whence $z=\lim z_l$ satisfies $z\in f^{-1}(\zeta )$ in contradiction with $z\in \partial \Omega$.
 
 Now, we set $h=f/(\varphi _{p_1}\cdots\varphi _{p_n})$, where $p_1,\dots,p_n \in\Omega$ are the roots of $f$. 
 The holomorphic function $h$ is not vanishing in $\Omega$, has modulus $1$ on $\Gamma _0$ and a constant modulus on each curve $\Gamma _i$. 
 Hence, 
 we may write $h/h(z_0)=g_\rho \in \G$. Noticing that $\log|f|=0$
 on the boundary, we get  $\rho _j=-\sum_{\ell=1}^n\omega_j(p_\ell)$, $1\leqslant j\leqslant k$. 
 
 c) Note that $\omega_j(p_\ell)>0$, thus $\omega_j(p_1)+\cdots+\omega_j(p_n)=-\rho _j\geqslant1$. 
 By summation on $j$ and using that $\omega_1+\cdots+\omega_k=1{-}\omega_0$, we get 
 $n-\omega_0(p_1)-\cdots-\omega_0(p_n)\geqslant k$, which implies $n>k$. 
\end{proof}

\noindent{\bf Remark. }{\it From this lemma, we deduce that the inner functions are holomorphic in $\overline\Omega$ and that, 
if they are not constant, they have at least $k{+}1$ zeros}.\medskip

It will be useful to also define $\varphi _p$ for $p$ on the boundary. We set
\begin{align}\label{eq2}
\varphi _p(z)=1,\quad \textrm{if }p\in \Gamma _0, \quad \varphi _p(z)=g_{e_i}(z),\quad \textrm{if }p\in \Gamma _i.
\end{align}
Here $\{e_i\}$ is the canonic basis of $\R^k$, $e_1=(1,0,\dots,0)^T$,\ \dots This definition is justified by the following lemma

\begin{lemma}
 Assume that the sequence of $p_n\in \Omega$ tends to $p\in\partial \Omega$ and set  $\overline\Omega _{p\varepsilon }:=\{z\in\overline\Omega\,: |z{-}p|\geqslant\varepsilon\} $, $\varepsilon >0$. Then, the sequence $|\varphi _{p_n}|$ tends to $|\varphi _p|$ in $C^\infty(\overline\Omega _{p\varepsilon })$
 and there exists a sequence of $\lambda _n\in \T$ such that $\lambda _n\varphi _{p_n}$ tends to $\varphi _p$ 
 in $C^\infty(\overline\Omega _{p\varepsilon })$.
\end{lemma}
\begin{proof}
 The Green functions $G_{p_n}$  are bounded in the Sobolev space $W^{1,1.5}_0(\Omega)$ and weakly converge (not converges)
  in this space\,; it holds
\[ 
\int_\Omega\nabla G_{p_n}.\nabla v\,dxdy=v(p_n)\to 0,
\]
for all smooth functions $v$ vanishing on the boundary. Thus, the weak limit is zero and $G_{p_n}$ tends to zero in $C^\infty(\overline\Omega _{p\varepsilon })$.
We also have $\omega_j(p_n)\to\omega_j(p)$, hence $c_i(p_n)=\sum_{j=1}^3n_{ij}\omega_j(p_n)\to 0$ if $p\in\Gamma _0$ and  $c_i(p_n)\to n_{ij}$ if $p\in \Gamma _j$. It then follows from \eqref{eq1} that (in $C^\infty(\overline\Omega _{p\varepsilon })$)
$\log|\varphi _{p_n}|\to 0$ if $p\in\Gamma _0$ and  $\log|\varphi _{p_n}|\to 2\pi \sum_{i=1}^3n_{ij}\,\omega_i=\log|g_{e_j}|=\log|\varphi _p|$ if $p\in \Gamma _j$, $j=1,2$ or $3$. This implies the convergence of  $|\varphi _{p_n}|$ to $|\varphi _p|$. 

The convergence of $ \log|\varphi _{p_n}|$ to $\log|\varphi _p|$ also induces the convergence of
partial derivatives of $\arg(\varphi _{p_n})-\arg(\varphi _p)$ to zero in $C^\infty(\overline\Omega _{p\varepsilon })$. 
Now, we choose $z_1\in \partial \Omega$, $z_1\neq p$ and set $\lambda _n\in\T$ such that $\arg(\lambda _n\varphi _{p_n}(z_1))=\arg(\varphi _p(z_1))$.
Then it holds $\arg(\lambda _n\varphi _{p_n})\to\arg(\varphi _p)$ in $C^\infty(\overline\Omega _{p\varepsilon })$ and $\lambda _n\varphi _{p_n}\to \varphi _p$.
\end{proof}

We deduce from the previous results
\begin{theorem}
 The inner functions $f$ in $\Omega$ of order $\leqslant N$ with $N\geqslant k{+}1$ are the functions of the form
 \begin{align*}
 f&=\lambda \,g_\rho \,\varphi _{p_1}\cdots \varphi _{p_N},\quad \textrm{with   }\lambda \in\T,\  \rho =(\rho _1,\dots,\rho _k)^T, \ \rho _j=-\sum_{\ell=1}^N\omega_j(p_\ell)\in\Z,
 \end{align*}
and $p_1\cdots p_N\in\overline \Omega$. They are holomorphic in $\overline\Omega$.
 \end{theorem}

Therefore, the inner functions of order $\leqslant N$ are defined by a compact set of parameters : 
$\lambda \in\T$, $\rho =(\rho _1,\dots,\rho _k)^T$,  $p_1\cdots p_N \in\overline \Omega$ (note that $|\rho _j|\leqslant N$). 
The set of these functions has the following property of sequential compactness:
from all sequences of inner functions in $\Omega$ with order $\leqslant N$, we can extract a subsequence
which is convergent to an inner function of order $\leqslant N$ in $C^\infty(K)$ for all compact sets $K\subset\Omega$.

\begin{lemma}\label{lem7}
For all given points $\zeta _1,\dots,\zeta _r\in \Omega$, there exist $k'\leqslant k$ points $\zeta _{r+1},\dots,\zeta _{r+k'}\in\Omega$ such that
\begin{align}\label{mesharm}
\omega_j(\zeta _1)+\cdots+\omega_j(\zeta _{k'+r})\in\N, \quad j=1,\dots,k.
\end{align}
Furthermore, these points $\zeta _{r+1},\dots,\zeta _{r+k'}$ may be chosen pairwise
distinct and distinct from $\zeta _1,\dots,\zeta _r$.
\end{lemma}
\begin{proof}
This result is clear if $k=1$;
the proof is given in the Appendix for the general case. 
\end{proof}
\begin{corollary}\label{cor8}
  With all given points $(\zeta _1,\dots,\zeta _r)\in \Omega^r$, there exists an inner function in $\Omega$ with order less or equal to $k{+}r$ vanishing at the points $\zeta _1,\dots,\zeta _r$. 
  \end{corollary}
\begin{proof}
It suffices to choose the function $ f=g_\rho  \,\varphi _{\zeta _1}\cdots \varphi _{\zeta _{k'+r}}$ with $\rho_j =-\sum_{\ell=1}^{k'+r}\omega_j(\zeta_\ell)$.
\end{proof}

\section{A Schwarz type lemma}\label{sch}
From now on $\EN$ denotes the set of inner functions in $\Omega$ of order $\leqslant N$ and we assume that $0\in \Omega$. 
We are concerned with the problem, for $N\geqslant k{+}1$ given,
\begin{align*}
(S_N)\quad\left\{\begin{matrix}
\textrm{find a function }h_N \in \EN\textrm{ which realizes} ,\hfill\\[10pt]
c_N=\max\{|h'(0)|\,: h\in \EN\}.
\end{matrix}
\right.
\end{align*}
We first remark that the set $\EN$ is nonempty; indeed, from Lemma\,\ref{lem7} and Corollary\,\ref{cor8}, 
there exists an inner function vanishing at $0$: $ h=g_\rho  \,\varphi_0 \varphi _{\zeta _1}\cdots \varphi _{\zeta _k}$
with $\zeta _j\neq 0$ for all $j$; furthermore, $0$ is a simple root of this function, thus $h'(0)\neq 0$, which shows that $c_N>0$.
Since the set of parameters $\lambda ,p_1,\dots,p_{N},\rho _1,\dots,\rho _N $ which define a function $h\in \mathcal E_N$ is compact, it is clear that a maximal function $h_N$ exists. Furthermore, the function
$h(z)=\frac{h_N(z)-h_N(0)}{1-\overline{h_N(0)}h_N(z)}\in \mathcal E_N$, whence we deduce $h_N(0)=0$ from $|h'(0)|= |{h_N}'(0)|/(1{-}|h_N(0)|^2)\leqslant |{h_N}'(0)|$.

We consider a solution $h_N\in\mathcal E_N$ of $(S_N)$ that we write, with $k\leqslant n\leqslant N{-}1$,
\begin{align*}
h_N(\zeta ){=}g_{\rho^*}(\zeta ) \varphi_0(\zeta )\varphi_{p _1^*}(\zeta )\dots\varphi_{p_n^*}(\zeta),\ \ \textrm{with }\omega_j(0){+}\sum_{\ell=1}^n\omega_j(p_\ell^*)=-\rho _j^*,\  j=1,\dots,k,
\end{align*}
with $p_j^*\in \Omega$ if $1\leqslant j\leqslant n$ and $\rho _j^*\in \Z$. In particular, we will have
$|h'(0)|\leqslant |{h_N}'(0)|$ for the functions $h$ of the form with $p_j\in \Omega$
\begin{align*}
h(\zeta )=g_{\rho^*}\!(\zeta ) \varphi_0(\zeta )\varphi_{p _1}(\zeta )\dots\varphi_{p_n}(\zeta),\textrm{such that }\omega_j(0){+}\sum_{\ell=1}^n\omega_j(p_\ell){=}{-}\rho _j^*,\  j{=}1,\dots,k,
\end{align*}
 with $p_j\in \Omega$.
 Note that $h'(0)=\varphi_0'(0 ) g_{\rho^*}(0 )\varphi_{p _1}(0)\dots\varphi_{p_n}(0)$, whence the
 inequality $|h'(0)|\leqslant |{h_N}'(0)|$ is equivalent to 
 $|\varphi_{p _1}(0 )\dots\varphi_{p_n}(0)|\leqslant |\varphi_{p_1^*}(0 )\dots\varphi_{p_n^*}(0)|$.
The logarithm function being increasing, we deduce from this inequality that 
\begin{align*}
&h_N \textrm{ realizes} ,\gamma^*=\displaystyle\max\{\log|\varphi _{p_1}(0)|{+}\cdots{+}\log|\varphi _{p_n}(0)|\}\hskip5cm\\
&\textrm{under the constraints }p_1,\dots,p_n\in \Omega\textrm{ and  }
\omega_j(0){+}\sum_{m=1}^n\omega_j(p_m)=-\rho _j^*,\  j=1,\dots,k.\
\end{align*}
We deduce from the necessary Fritz John optimality conditions\,\cite{mafr} : there exist $k$ real numbers $\ell_0,\dots,\ell_{k-1}$ not all zero such that
\begin{align*} 
\ell_0\nabla\big(\log|\varphi _{p_1^*}(0)|{+}\cdots{+}\log|\varphi _{p_n^*}(0)|\big)+\sum_{j=1}^{k-1}\ell_j\sum_{m=1}^n\nabla\omega_j(p_m^*)=0.
\end{align*}
We now write $p_j=x_j{+}i\, y_j$, with $x_j$ and $y_j\in\R$, $\nabla_j=(\partial _{x_j};\partial _{y_j})$.
We note that the components of the gradient vector are $\nabla=(\partial _{x_1};\partial _{y_1};\cdots;\partial _{x_n};\partial _{y_n})$. This allows to rewrite the previous optimality condition as
\begin{align}\label{joh1}
\ell_0\nabla_m\log|\varphi _{p_m^*}(0)|+\sum_{j=1}^{k}\ell_j\nabla_m\omega_j(p_m^*)=0,\quad m=1,\dots,n.
\end{align}
We now introduce the real-valued harmonic function
\[
H(z)=\ell_0 \log|\varphi _z(0)|+\sum_{j=1}^k\ell_j\,\omega_j(z),
\]
and we set, with $z=x{+}iy$,
\begin{align} \label{a1}
a(z)=\partial_x H(z)-i\,\partial_y H(z). 
\end{align}
With this notation, the condition \eqref{joh1} becomes $a(p_m^*)=0$, $m=1,\dots,n$.\medskip

\noindent{\bf Remark.} {\it The function $a$ is not identically equal to zero in $\Omega$. Indeed, otherwise $H$ 
would be constant, which implies  $ \ell_0=0$ to eliminate the singularity in $0$, and then $\ell_j=0$, $j=1,\dots,k$, 
since the $\omega_j$ are independent; this will be in contradiction with the fact that the multipliers are not all zero. Note also that $a$ 
is analytic in a neighborhood of the boundary, therefore $a$ has only a finite number of zeros on the boundary.}\medskip

The function $H(\cdot)$ is harmonic in a neighborhood of $\Omega$, except at the point $0$, whence the function $a$ is meromorphic in $\overline\Omega$ with roots $p_j^*$, $j=1,\dots,n$, and a simple pole $0$. Note that $H(z)$ is constant on each component of $\partial \Omega$; 
cf. \eqref{eq2}. By differentiation of $H$ with respect to the arclength $\sigma $ at the point $z(\sigma )=x(\sigma ){+ }i\,y(\sigma )$ we get
\[
0=\frac{d}{d\sigma }H(z(\sigma ))=\frac{\partial H}{\partial x}\frac{dx}{d\sigma }+\frac{\partial H}{\partial y}\frac{dy}{d\sigma }=\Re \big(a(z(\sigma ))z'(\sigma )\big).
\]
We still differentiate this relation with respect to $\sigma $ and divide the result by the imaginary number $a(z)z'$; this gives
\[
\Im \Big(\frac{a'(z(\sigma ))z'(\sigma )}{a(z(\sigma ))}+\frac{z''(\sigma )}{z'(\sigma )}\Big)=0,\quad\textrm{whence  }\ \frac{1}{2\pi i}\int_{\partial \Omega}\frac{a'(z(\sigma ))z'(\sigma )}{a(z(\sigma ))}d\sigma =k{-}1.
\]
From the argument principle, the number of zeros of $a$ in $\Omega$ plus half the number of zeros on the boundary minus the number of poles is equal to $k{-}1$; thus $a$ has at most $k$ zeros in $\Omega$, consequently $n\leqslant k$ whence $n=k$ and $C_N=C_{k+1}$ is independent of $N$.

We admit for the time being the result of Section\,\ref{picder} and deduce the following result

\begin{theorem}
 The problem 
\begin{align*}
(S)\quad\left\{\begin{matrix}
\textrm{find a function }h_0 \in H^\infty(\Omega) \textrm{ with } \|h_0\|_\infty\leqslant 1 \textrm{ which realizes} ,\hfill\\[10pt]
h'_0(0)=\max\{|h'(0)|\,: h\in H^\infty(\Omega), \|h\|_\infty \leqslant 1\},
\end{matrix}
\right.
\end{align*} 
has a unique solution. Furthermore, this solution is an inner function of order $k{+}1$.
\end{theorem}
\begin{proof} For the existence of $h_0$, we can use the same proof as for the solution $h_N$ of $(S_N)$; it is easily seen that $\|h_0\|_\infty=1$. Then, such a solution $h_0$ being given, Theorem\,\ref{thderiv} shows the existence of 
$f_*\in H^{\infty}(\Omega)$ which realizes
\[
m^*=\min\{\|f\|_\infty\,: f\in H^\infty(\Omega), \ f^{(s)}(0)=h_0^{(s)}(0), \ 0\leqslant s\leqslant k\},
\]
and that $f_*/m^*$ is inner of order $\leqslant 2k$. We have clearly $m^*\leqslant \|h_0\|_\infty=1$ and $f_*'(0)=h_0'(0)$; we cannot have $m^*<1$, since otherwise $h=f_*/m^*$ would satisfy $\|h\|_\infty=1$ and $h'(0)>h_0'(0)$. 
Therefore, $m^*=1$ and $f_*$ is inner;  it is of order $k{+}1$ since it is solution of $(P_N)$, and it is solution of $(P)$.
From the uniqueness of $f_*$ which realizes $m^*$, we deduce that $f_*=h_0$, whence all the solutions of $(S)$ are inner;
thus, they belong to $C^\infty(\overline\Omega)$.

Now, if $f_1$ and $f_2$ are two solutions of $(S)$, then $\frac12(f_1{+}f_2)$ is also solution of $(S)$,
whence on the boundary of  $\Omega$ we have $|f_1|=|f_2|=\frac12|f_1{+}f_2|=1$, which implies
$f_1=f_2$ on the boundary; thus, $f_1=f_2$ in $\Omega$. 
\end{proof}

\section{An optimization problem with interpolation in $k$ points.}\label{pn}

From now on $\EN$ denotes the set of inner functions in $\Omega$ of order $\leqslant N$ and we are concerned with the problem: given $r$ 
different points $\zeta_1, \zeta _2,\dots,\zeta _r$ in $\Omega$ and $r$ complex numbers $\tau _1, \tau_2,\dots,\tau_r$ not all zero,
\begin{align*}
(P_N)\quad\left\{\begin{matrix}
\textrm{find a function }h_N \in \EN\textrm{ which realizes} ,\hfill\\[10pt]
c_N=\max\{c\in\R\,: \exists\, h\in \EN,\  h(\zeta _1)=c\,\tau _1,\dots, h(\zeta _r)=c\,\tau _r\}.
\end{matrix}
\right.
\end{align*} 
It is clear that, if the solution $f^*$ of problem (P) can be written in the form 
$f^*=m^*h_N$ with $h_N$  inner function of order 
$n\leqslant N $, then $h_N$ is solution of $(P_N)$ and it holds $c_N=1/m^*$. 
\begin{lemma}\label{lem4}
We assume that $N\geqslant k{+}r$. Then, problem $(P_N)$ has a solution; 
if furthermore $N\geqslant 2 k{+}r$, then $c_N$ is strictly positive.
\end{lemma}
\begin{proof}
Note that $\mathcal E_{r+k}\subset \EN$. Since the set of parameters $\lambda ,p_1,\dots,p_{N},\rho _1,\dots,\rho _N $ which define a function $h\in \mathcal E_N$ is compact, it suffices to show that there exists some $h\in \mathcal E_{r+k}$ which realizes the constraints $h(\zeta _l)=c\,\tau _l, \ l=1,\dots,r$. 
This is the case for $c=0$ since, from Corollary\,\ref{cor8}, there exist $\zeta _{r+1},\dots,\zeta _{r+k'}\in\Omega$ such that the inner function of order 
$\leqslant N$
\[
h(z)=g_\rho \varphi_{\zeta _1}\cdots\varphi_{\zeta _{r+k'}}, \textrm{with }\rho _j=-\sum_{\ell=1}^{r+k'}\omega_j(\zeta_\ell)\in\Z,\quad  j=1,\dots,k,
\]
vanishes at the points $\zeta _1,\dots,\zeta _r$. Thus the maximum is realized and $c_N\geqslant 0$.
\medskip

Now, we consider the mapping $A$ from $\R^{2k}=\C^k$ with values in $\R^k$ defined by
\[
A_j(p_{r+1},\dots,p_{r+k})=\sum_{\ell=r+1}^{r+k}\omega_j(p_l),\quad j=1,\dots,k.
\]
We assume for the time being that we can find $k{+}k'$ points $\zeta_{r+1},\dots,\zeta_{r+k+k'}\in\Omega$ such that:
\[
\sum_{\ell=1}^{r+k+k'}\omega_j(\zeta_\ell)\in\N,\quad j=1,\dots,k,
\]
the range of the derivative of $A$ at the point $(\zeta_{r+1},\dots,\zeta_{r+k})$ is equal to $k$ and
the points $\zeta _j$, $1\leqslant j\leqslant r{+}k{+}k'$ are all distinct. 

We keep the points $\zeta _{r+k+1},\dots,\zeta _{r+k+k'}$ fixed 
and $\rho _j=- \sum_{\ell=1}^{r+k+k'}\omega_j(\zeta_\ell)$. We set
\begin{align*}
\Phi _{j}(c;p_1,\dots,p_{r+k})&=\varphi_{p_1}(\zeta_j)\cdots\varphi_{p_{r+k}}(\zeta _j)
\varphi_{\zeta_{r+k+1}}(\zeta_j)\cdots\varphi _{\zeta _{r+k+k'}}(\zeta _j)-c\,\tau_j/g_\rho (\zeta _j),\\
&\phantom{5}\hskip5cm 1\leqslant j\leqslant r,\\
\Phi _{j+r}(c;p_1,\dots,p_{r+k})&=\sum_{\ell=1}^{r+k}\omega_j(p_\ell)+\sum_{\ell=r+k+1}^{r+k+k'}\omega_j(\zeta_\ell)+\rho_j,\quad j=1,\dots,k.
\end{align*}
Clearly, it holds $\Phi(0;\zeta _1,\dots,\zeta _{r+k})=0$. Now we look at the dependence of the first $r$ components 
of $\Phi (\cdot)$ w.r.t.  the first $r$ variables (regarded as mapping from $\R^{2r}$ into $\R^{2r}$). The partial derivatives exist and satisfy 
\begin{align*}
D_{ p_\ell}\Phi _j(0;\zeta _1\dots,\zeta _{r+k})&=0,\quad\textrm{if } \ell\neq j,\ 1\leqslant j\leqslant r,\\
D_{ p_j}\Phi _j(0;\zeta _1\dots,\zeta _{r+k})&\textrm{ is invertible},\quad\textrm{ if  } 1\leqslant j\leqslant r.
\end{align*}
Indeed, $\zeta _j$ is a simple root of $\varphi _{\zeta _j}(z)$ and $\varphi _{\zeta _j}(\zeta _\ell)\neq 0$ if $\ell\neq j$.
The derivative of the last $k$ components of $\Phi(0;\zeta _1\dots,\zeta _{r+k})$ with respect to the last $k$
variables  (considered as acting from $\R^{2k}$ into $\R^k$) has rank $k$. Therefore, 
for $c=0$, the derivative of $\Phi $ from $\C^{r+k}=\R^{2r+2k}$ into $\R^{2r+k}$ has rank $2r{+}k$ at the point $(\zeta _1\dots,\zeta _{r+k})$; then, we deduce from the implicit function theorem the existence, for small $c$, of a solution to the equations $\Phi(c;p_1,\dots,p_{r+k})=0$.
For this solution, the function
\[
h(z)=g_\rho(z) \varphi_{p_1}(z)\cdots\varphi_{p_{r+k}}(z)
\varphi_{\zeta_{r+k+1}}(z)\cdots\varphi _{\zeta _{r+k+k'}}(z)
\]
satisfies $h(\zeta _j)=c\tau _j$ and belongs to $\mathcal E_{r+2k}$, whence $c_{r+2k}>0$.
 \medskip

Now, we turn to the existence of $k{+}k'$ points $\zeta_{r+1},\dots,\zeta_{r+k+k'}\in\Omega$ such that : 
\[
\sum_{\ell=1}^{r+k+k'}\omega_j(\zeta_\ell)\in\N,\quad j=1,\dots,k,
\]
the range of the derivative of $A$ at the point $(\zeta_{r+1},\dots,\zeta_{r+k})$ is equal to $k$ and
the points $\zeta _j$, $1\leqslant j\leqslant r{+}k{+}k'$ are all distinct. 

For that, we first choose $k$ distinct points $\zeta_{r+1},\dots,\zeta_{r+k}\in\Omega$, also distinct from $\zeta_1,\dots,\zeta_r$, such that the rank of the
derivative of $A$ at the point $(\zeta_{r+1},\dots,\zeta_{r+k})$ be of rank $k$, (it is easy to find such points). From Lemma\,\ref{lem7}, there exist points $\zeta_{r+k+1},\dots,\zeta_{r+k+k'}\in\Omega$ such that $\rho _j=-\sum_{\ell=1}^{r+k+k'}\omega_j(\zeta_\ell)\in\Z,\  j=1,\dots,k$,
and such that the points $\zeta_1,\dots,\zeta_{r+k+k'}$ are all distinct. 
\end{proof}

We consider a solution $h_N\in\mathcal E_N$ of $(P_N)$ that we write as
\begin{align}\label{gstar}
h_N(\zeta )=e^{i\theta_N}g_{\rho^*}(\zeta ) \varphi_{p _1^*}(\zeta )\dots\varphi_{p_N^*}(\zeta),\quad\textrm{with }\sum_{\ell=1}^N\omega_j(p_\ell^*)=-\rho _j^*,\  j=1,\dots,k,
\end{align}
 with $p_j^*\in \Omega$ if $1\leqslant j\leqslant n$ and $p_j^*\in\partial \Omega$ if $j>n$, $\rho _j^*\in \Z$.
If some $\tau _j$ are zero, we assume that the numbering is chosen such that $\tau _j\neq 0$ if $1\leqslant j\leqslant r_1$ and $\tau _j=0$ if $r_1<j\leqslant r$, otherwise $r_1=r$. We also assume that the representation
 \eqref{gstar} is such that $p^*_{n-j}=\zeta_{r-j}$ if $0\leqslant j<r{-}r_1$. The logarithm function being increasing, we can rewrite problem $(P_N)$ in the equivalent form 
\begin{align*}
\left\{\begin{matrix}
\textrm{find }h_N \in \EN\textrm{ which realizes} ,\hfill\\[10pt]
\gamma_N=\displaystyle\max_{\gamma \in\R,\,h\in\mathcal E_N}\{\gamma \,: \ \log(| h(\zeta _j)/ \tau _j|){-}\gamma =0,\ \arg( h(\zeta _j)/\tau _j)=0 ,\   1\leqslant j\leqslant r_1,\\
\hfill h(\zeta _j)=0,\ r_1<j\leqslant r \}.
\end{matrix}
\right.
\end{align*}
Now, we introduce $\e=(e_1,\dots,e_{n_1})\in\C^{n_1}$, $n_1=n{+}r_1{-}r$, $\tv=(t_{n+1},\cdots,t_N)$  and we set
\begin{align*}
d_\ell(\gamma,\theta,\e,\tv)&:=\log(| h(\zeta_\ell)/\tau_\ell|){-\gamma },\\
d_{\ell+r_1}(\gamma,\theta,\e,\tv)&:=\arg( h(\zeta _\ell)/\tau _\ell),\quad \ell=1,\dots,r_1,\\
d_{2r_1+j}(\gamma,\theta, \e,\tv)&:={\scriptstyle\sum_{\ell=1}^N}\omega_j(p_\ell)+\rho _j^*,\  j=1,\dots,k,
\end{align*}
where
\begin{align*}
h(\zeta )&=e^{i\theta }g_{\rho^*}(\zeta )\, \varphi_{p _1}(\zeta )\dots\varphi_{p_N}(\zeta),\\
p_j&=p_j^*{+}e_j \ \textrm { if }\ 1\leqslant j\leqslant n_1,\  p_j=p_j^* \ \textrm { if }\ n_1< j\leqslant n, \quad p_j=p_j^*{-}t_j\nu_j^* \ \textrm { if } j>n.
\end{align*}
Here, $\nu_j^*$ denotes the unit outward normal vector at the boundary point $p_j^*\in\partial \Omega$.
There exists $\sigma >0$ such that, if $|e_j|<\sigma $ and if $0\leqslant t_j<\sigma $ for all $j$, then $p_j\in \overline{\Omega}$. If furthermore $d_{2r_1+j}(\gamma,\theta, \e,\tv)=0$, then $h\in \EN$.
Therefore $(\gamma _N,\theta^*,\mathbf 0,\mathbf 0)$ realizes
\begin{align*}
\gamma _N=\max\{\gamma\,:\  d_m(\gamma,&\theta,\e,\tv)=0,\ 1\leqslant m\leqslant 2r_1{+}k,\\
&\gamma,\theta  \in\R,\quad |e_j|<\sigma \ \textrm{for }1\leqslant j\leqslant n_1, 0\leqslant t_j<\sigma \textrm{ if } j>n\}.
\end{align*}
Note that the functions $\log(\varphi _p(\zeta ))$ are smooth if $p\neq \zeta $ and that $p_j^*\neq \zeta _l$  for $j\leqslant n_1$, $j\geqslant n$ and all $l\leqslant r_1$.
So, the constraints $d_m(\gamma,\theta,\e,\tv)$ are smooth w.r.t. $e_j$ and $t_j$ for sufficiently small $\sigma $. We deduce from the necessary Fritz John optimality conditions\,\cite{mafr} : there exist $2r_1{+}k{+}1$ real numbers $\ell_0,\dots,\ell_{2r_1+k}$ and $N{-}n$ positive or null numbers $x_{n+1},\dots,x_N$, not all zero, such that
\begin{align*} 
\ell_{0}\nabla \gamma +\sum_{m=1}^{2r_1+k}\ell_m\nabla d_m(\gamma _N,\theta_N,\mathbf 0,\mathbf 0)
+\sum_{j=n+1}^Nx_j\nabla t_j=0.
\end{align*}
We now write $e_j=x_j{+}i\, y_j$, with $x_j$ and $y_j\in\R$ and $\nabla_j=(\partial _{x_j};\partial _{y_j})$.
We note that the $2{+}2n_1{+}N{-}n$ components of the gradient vector $\nabla$ have the following order: first the partial derivatives $\partial_ \gamma $, $\partial _\theta$, then the components $\nabla_j$, $j=1,\dots,n_1$, and finally the $\partial _{t_j}$, $n{+}1\leqslant j\leqslant N$. Note that $\partial _\theta d_{j+r_1}(\gamma ,\theta ,\e,\tv)=1$, if $1\leqslant j\leqslant r_1$, $\partial _\theta d_m(\gamma ,\theta ,\e,\tv)=0$ otherwise. This allows to rewrite the previous optimality condition as
\begin{align}
\ell_0-\ell_1-\cdots-\ell_{r_1}&=0, \quad \ell_{r_1+1}+\cdots+\ell_{2r_1}=0,  \nonumber\\ 
\sum_{m=1}^{2r_1+k}\ell_m\nabla_j d_m(\gamma _N,\theta_N,\mathbf 0,\mathbf 0)&=0, \quad j=1,\dots,n_1, \label{la}\\
\sum_{m=1}^{2r_1+k}\ell_m\partial_{t_j}  d_m(\gamma _N,\theta_N,\mathbf 0,\mathbf 0)&=-x_j, \quad j=n{+}1,\dots,N.
\label{joh}
\end{align}
Notice that this implies that $\ell_1,\dots,\ell_{2k+r_1}$ are not all zero. Simple calculations give 
\begin{align}
\nabla_{\!j \,}d_m(\gamma _N,\theta_N,\mathbf 0,\mathbf 0)&=\nabla_{\!j \,}\log|\varphi _{p_j^*}(\zeta _m)|\nonumber\\
\nabla_{\!j \,}d_{r_1+m}(\gamma _N,\theta_N,\mathbf 0,\mathbf 0)&=\nabla_{\!j \,}\arg \varphi_{p_j^*}(\zeta_m),\ m=1,\dots,r_1,\label{graj}\\
\nabla _{\!j \,}d_{2r_1+m} (\gamma _N,\theta_N,\mathbf 0,\mathbf 0)&=\nabla_{\!j \,}\omega_m (p_j^*),\ m=1,\dots,k.\nonumber
\end{align}
We now introduce the real multi-valued harmonic function
\begin{align} \label{H}
H(z)=\sum_{m=1}^{r_1}\big(\ell_m\,\log|\varphi _z(\zeta _m)|+\ell_{m+r_1}\arg\varphi_z(\zeta _m)\big)+\sum_{m=1}^k\ell_{2r_1+m}\omega_m(z),
\end{align}
and we set, with $z=x{+}iy$,
\begin{align} \label{a}
a(z)=\partial_x H(z)-i\,\partial_y H(z). 
\end{align}
With this notation, condition \eqref{la} becomes $a(p_j^*)=0$, $j=1,\dots,n_1$.\medskip

\noindent{\bf Remark.} {\it The function $a$ is not identically equal to zero in $\Omega$. Indeed, otherwise $H$ would be constant, 
which implies  $ \ell_m=\ell_{m+r_1}=0$ to eliminate the singularity at $\zeta _m$, $m=1,\dots,r$, and then $\ell_{2r_1+m}=0$ 
since the $\omega_m$ are independent. This contradicts  that these multipliers are not all zero. Note also that $a$ is analytic in a 
neighborhood of the boundary; therefore, $a$ has only a finite number of zeros on the boundary.}\medskip

The function $H(\cdot)$ is harmonic in a neighborhood of $\Omega$, except at the points $\{\zeta _1,\dots,\zeta _k\}$, whence the function $a$ is meromorphic in $\overline\Omega$ with roots $p_j^*$, $j=1,\dots,n_1$, and possible poles $\zeta _1,\dots, \zeta _{r_1}$. These poles are simple, since $\zeta _j$ is a simple root of $\varphi _{z}(\zeta _j)$. Note that $H(z)$ is constant on each component of $\partial \Omega$; cf.\ \eqref{eq2}. 
By differentiation of $H$ with respect to the arclength $\sigma $ at the point $z(\sigma )=x(\sigma ){+ }i\,y(\sigma )$ we get
\[
0=\frac{d}{d\sigma }H(z(\sigma ))=\frac{\partial H}{\partial x}\frac{dx}{d\sigma }+\frac{\partial H}{\partial y}\frac{dy}{d\sigma }=\Re \big(a(z(\sigma ))z'(\sigma )\big).
\]
We still  differentiate  this relation with respect to $\sigma $ and divide the result by the imaginary number $a(z)z'$;  this gives
\[
\Im \Big(\frac{a'(z(\sigma ))z'(\sigma )}{a(z(\sigma ))}+\frac{z''(\sigma )}{z'(\sigma )}\Big)=0,\quad\textrm{whence  }\ \frac{1}{2\pi i}\int_{\partial \Omega}\frac{a'(z(\sigma ))z'(\sigma )}{a(z(\sigma ))}d\sigma =k{-}1.
\]
From the argument principle, the number of zeros of $a$ in $\Omega$ plus half the number of zeros on the boundary minus the number of poles is equal to $k{-}1$; thus $a$ has at most $r_1{+}k{-}1$ zeros in $\Omega$;  consequently,  $n_1\leqslant r_1{+}k{-}1$,  
whence $n\leqslant r{+}k{-}1$.\medskip

We now turn to condition \eqref{joh} associated with the inequality constraints $t_j\geqslant 0$ 
at the points $p_j^*=z(\sigma _j)$, $j>n$. 
It can be written as $\Re(a(z(\sigma _j))\nu_j^* )=x_j$\,; but  $z'(\sigma _j)=i\nu_j^*$ and we have seen 
that $\Re(a(z(\sigma _j))z'(\sigma _j)) )=0$.
Thus, condition \eqref{joh} becomes $i\,a(z(\sigma _j)\,z'(\sigma _j)=-x_j$ and implies
\begin{equation}\label{bord}
i\,z'(\sigma _j)a(z(\sigma _j))\le 0\quad\textrm{ if } z(\sigma _j)=p_j^*, n<j\leqslant N.
\end{equation}\smallskip

At this stage, we do not know if our solution of Problem $(P_N)$ depends or not
on $N$, but it is clear that, if $h_N$ is a solution of $(P_N)$, then $h_N\in \mathcal E_{r+k-1}$ and $h_N$ is solution of $(P_{N'})$ for all $r{+}k{-}1\leqslant N'\leqslant N$\,; hence, $c_N=c_{r+k-1}$ and $\gamma _N=\gamma _{r+k-1}$ are independent of $N$. From now on, we set $n=r{+}k{-}1$, $c^*=c_n$, $\gamma ^*=\gamma_n$ and we consider $h_N=h_n$ a solution of $ P_n$ (which is also a solution of $(P_N)$ for all $N\geqslant n$).
But, even if $h_N$ is the same for all $N\geqslant n$,  its representation
\eqref{gstar} depends on the choice of $p_{n+1}^*,\dots,p_N^*$ (which can be made arbitrary on the boundary) and, a priori,
the multipliers $\ell_m$ (and then the functions $H$ and $a$) depend, not only of $N$, but also of this choice.
\medskip

To the choice of  $h_N$ with roots $p_1^*,{\dots},p_n^*$ we add a sequence of points $p_{n+1}^*,{\dots},p_{n+l}^*,{\dots}$ that we assume dense in $\partial \Omega$.
For each value of $N\geqslant n$, we write $h_N$ as
\begin{align*}
h_N(\zeta )=e^{i\theta_{N}}g_{\rho_{N}}(\zeta ) \varphi_{p _1^*}(\zeta )\dots\varphi_{p_N^*}(\zeta),\quad\textrm {with }\sum_{\ell=1}^N\omega_j(p_\ell^*)=-\rho _j^*,\  j=1,\dots,k,
\end{align*}
and obtain as previously multipliers $\ell_{1,N},\dots,\ell_{2r_1+k,N}$,  $x_{n+1,N},\dots,x_{N,N}$, not all zero, such that
\begin{align*}
\sum_{m=1}^{2r_1+k}\ell_{m,N}\nabla_j d_m&=0, \quad j=1,\dots,n_1,\\
\sum_{m=1}^{2r_1+k}\ell_{m,N}\,\partial_{t_j}  d_m&=-x_{j,N}\leqslant0, \quad j=n{+}1,\dots,N.
\end{align*}
(We have written $\nabla_jd_m$ in place of $\nabla_jd_m(\gamma^*,\theta_N,\mathbf 0,\mathbf 0)$ since, cf. \eqref{graj}, this term does not depend on $\gamma^*$, on $\theta_N$, or on $N$).
Multiplying, if needed, each multiplier by a positive constant, we may assume that $\max\{|\ell_{1,N}|,\dots,|\ell_{2r_1+k,N}|\}=1$. 
Then, there is a subsequence
$\ell_{1,N_m},\dots,\ell_{2r_1+k,N_m}$ which tends to $\ell_1^*,\dots,\ell_{2r_1+k}^*$ as $m\to\infty$ and we get
\begin{align*}
\sum_{m=1}^{2r_1+k}\ell_m^*\nabla_j d_m=0, \quad  \textrm{if  }1\leqslant j\leqslant n_1,\qquad
\sum_{m=1}^{2r_1+k}\ell_m^*\,\partial_{t_j}  d_m\leqslant 0, \quad \textrm{if  }j>n.
\end{align*}
Now we set
\begin{align*} 
H_N(z)=\sum_{m=1}^{r_1}\big(\ell_m^*\,\log|\varphi _z(\zeta _m)|+\ell_{m+r_1}^*\arg\varphi_z(\zeta _m)\big)+\sum_{m=1}^k\ell_{2r_1+m}^*\omega_m(z),
\end{align*}
 and $a^*(z)=\partial_x H_N(z)-i\,\partial_y H_N(z)$. 

As previously $a^*$ is meromorphic in $\overline\Omega$, not identically zero, $a^*(p_j^*)=0$ for $j=1,\dots, n_1$, the possible poles are $\{\zeta _1,\cdots,\zeta _k\}$ and are simple. Furthermore, at the boundary points $p_j^*=z(\sigma _j)$, $j>n$  it holds
$i\,z'(\sigma _j)a^*(z(\sigma _j))\le 0$\,; since these points are dense in $\partial \Omega$ we obtain 
$i\,z'(\sigma )a^*(z(\sigma ))\leqslant0$ on $\partial \Omega$.\medskip

I think that the multiplicity of $p_j^*$, $1\leqslant j\leqslant n_1$ as root of $h_N$ is less or equal than its multiplicity as root of $a^*$. Clearly, this is the case if the $p_j^*$ are all distinct. We now look at the case of multiple roots.
Without loss of generality, we assume that in the numbering $1\leqslant j\leqslant n_1$,  equal $p_j^*$s are successive. For a root $p_{j+1}^*=\cdots=p_{j+\nu_j}^*$ with multiplicity $\nu_j$, we define $u_{j+m}=e_{j+1}^m+\cdots+e_{j+\nu_j}^m$, for $m=1,\dots,\nu_j$\,; note that, if $m>\nu_j$, then $e_{j+1}^m+\cdots+e_{j+\nu_j}^m$ are polynomial functions of $u_{j+1},\dots, u_{j+\nu_j}$.  In this way, we get a mapping between $\e$ and $\uu=(u_1,\dots,u_{n_1})$ which is one to one up to a renumbering of the equal $e_j$. It holds
\begin{align*}
d_m(\gamma,\theta,\e,\tv)-d_m(\gamma,\theta,\mathbf 0,\tv)&{=}\sum_{j=1}^{n_1}(\log|\varphi _{p_j^*+e_j}(\zeta _m)|-\log|\varphi _{p_j^*}(\zeta _m|), \\
d_{m+r_1}(\gamma,\theta,\e,\tv){-}d_{m+r_1}(\gamma,\theta,\mathbf 0,\tv)&{=}\sum_{j=1}^{n_1}(\arg(\varphi _{p_j^*+e_j}(\zeta _m)){-}\arg(\varphi _{p_j^*}(\zeta _m )), 1\leqslant m\leqslant r_1,\\
\hskip-1.5cm
d_{m+2r_1}(\gamma,\theta,\e,\tv){-}d_{m+2r_1}(\gamma,\theta,\mathbf 0,\tv)&{=}\sum_{j=1}^{n_1}(\omega_m(p_j^*{+}e_j){-}\omega_m(p_j^*)), \ 1\leqslant m\leqslant k.
\end{align*}
These functions $d_\ell(\gamma,\theta,\e,\tv)$ are the real parts of holomorphic functions w.r.t. $\e_j$ so, 
we can write the expansions  (with $\nu$ integer)
\begin{align}\label{dev}
d_m(\gamma,\theta,\e,\tv)=d_m(\gamma,\theta,\mathbf 0,\tv)+\Re\Big(\sum_{j=1}^{n_1}\sum_{\nu\geqslant1}d_{m,\nu}(p_j^*)e_j^\nu\Big),
\end{align}
which are convergent if $|e|\leqslant\sigma $ for some $\sigma >0$. This allows to define
\[
\delta _\ell(\gamma,\theta,\uu,\tv):=d_\ell(\gamma,\theta,\e,\tv),
\]
and this function $\delta _\ell(\gamma,\theta,\uu,\tv)$ is the real part of a holomorphic function in each $u_j$.
With this change of variables, $(\gamma^*,\theta^*,\mathbf 0,\mathbf 0)$ realizes 
\begin{align*}
\gamma^*=\max\{\gamma\,:\  \delta _\ell(\gamma,&\theta,\uu,\tv)=0,\ 1\leqslant\ell\leqslant2r_1{+}k,\\
&\gamma,\theta  \in\R,\quad |\uu_j|<\sigma' \ \textrm{pour }1\leqslant j\leqslant n_1, 0\leqslant t_j<\sigma' \textrm{ if } j>n\}.
\end{align*}
As before, we write the Fritz John necessary optimality conditions : there exists $2r_1{+}k{+}1$ real numbers $\ell_0^*,\dots,\ell_{2r_1+k}^*$ 
and positive or null numbers $x_{n+1}^*,\dots$, not all zero, such that 
\begin{align*} 
\ell_{0}^*\nabla \gamma +\sum_{m=1}^{2r_1+k}\ell_m^*\nabla \delta_m(\gamma^*,\theta^*,\mathbf 0,\mathbf 0)
+\sum_{j\geqslant n+1}x_j^*\nabla t_j=0.
\end{align*}
The change with respect to our previous calculations is that, from now on, we use the notation $\nabla_j=(\partial _{x_j};\partial _{y_j})$ for the partial derivatives with respect to $u_j=x_j{+}i y_j$.
So, the optimality conditions rewrite
\begin{align}
\ell_0^*-\ell_1^*-\cdots-\ell_{r_1}^*&=0, \quad \ell_{r_1+1}^*+\cdots+\ell_{2r_1}^*=0,  \nonumber\\ 
\sum_{m=1}^{2r_1+k}\ell_m^*\nabla_j \delta _m(\gamma^*,\theta^*,\mathbf 0,\mathbf 0)&=0, \quad j=1,\dots,n_1, \label{lab}\\
\sum_{m=1}^{2r_1+k}\ell_m^*\partial_{t_j}  \delta _m(\gamma^*,\theta^*\mathbf 0,\mathbf 0)&=-x_j^*, \quad j\geqslant n{+}1.
\label{john}
\end{align}
As before we introduce
\begin{align*} 
h_N(z)=\sum_{m=1}^{r_1}\big(\ell_m^*\,\log|\varphi _z(\zeta _m)|+\ell_{m+r_1}^*\arg\varphi_z(\zeta _m)\big)+\sum_{m=1}^k\ell_{2r_1+m}^*\omega_m(z),
\end{align*}
and $a^*(z)=\partial_x h_N(z)-i\,\partial_y h_N(z)$. 
We deduce from \eqref{john} the property $i\,z'(\sigma )a^*(z(\sigma ))\leqslant0$ on the boundary $\partial \Omega$.

Notice that an expansion analogous to \eqref{dev} exists in each point $z=p_j^*$
\[
h_N(z{+}e)=h_N(z)+\sum_{m=1}^{2r_1+k}\sum_{\nu\geqslant 1}\ell^*_m\Re(d_{m,\nu}(z)e^\nu),
\]
whence
\[
a^*(z{+}e)=\sum_{\nu\geqslant1}\nu\,\ell^*_m\sum_{m=1}^{2r_1+k}(d_{m,\nu}(z)e^{\nu-1}).
\]
Now, if $p_{j+1}^*=\cdots=p_{j+\nu_j}^*$ is a root of $h_N$ with multiplicity $\nu_j$,
condition \eqref{lab} corresponding to the index $j{+}\nu$ with $1\leqslant\nu\leqslant\nu_j$ writes $\sum_{m=1}^{2r_1+k}\ell_m^*d_{m,\nu}(p_{j+1}^*)=0$\,; this implies $a^*(p_{j+1}^*{+}e)=O(e^\nu_j)$, that is to say: $p_{j+1}^*$ is a root of $a^*$ of multiplicity $\nu_j$.

 \section{Interpolation with minimal maximum norm}\label{picnev}

\begin{theorem}\label{th5}
 There exists a unique holomorphic function $f^*\in H^\infty(\Omega)$ realizing
  \[
m^*=\min\{\|f\|_\infty\,: f\ \in H^\infty(\Omega),\  f(\zeta _1)=\tau _1,\dots,f(\zeta _k)=\tau _k\}.
\]
Furthermore, $h_N=f^*/m^*$ is an inner function of order $n\leqslant r{+}k{-}1$ and it is the unique solution 
of Problem $(P_N)$ for all $N\geqslant n$. 
\end{theorem}

\begin{proof}
 Let $h_N$ be a common solution of Problems $(P_N)$, $N\geqslant r{+}k{-}1$,  obtained in the previous section 
 and $a^*$ the corresponding meromorphic function. 
 
Let $f$ be a holomorphic function $f$ in $\Omega$ which satisfies $f(\zeta _l)=\tau _l$, $l=1,\dots,k$; 
we set
 \[
\varphi (z):= ia^*(1{-}c^*f/h_N)=i\frac{(h_N-c^*f)a^*}{h_N}.
\]
This function $\varphi $ is holomorphic in $ \Omega$; indeed, the zeros $p_j^*$ of $h_N$, $1\leqslant j\leqslant n_1$,  
do not create poles since their multiplicities are less than or equal to the 
corresponding multiplicity as root of $a^*$, the other possible zeros 
$p_{n-j}^*=\zeta_{r-j} $, $n_1<j\leqslant n$ are simple and are root of $h_N{-}c^*f$; furthermore the other possible poles
 $\zeta _l$ of $a^*$ are simple and roots of $h_N{-}c^*f$.
Using the fact that $i z'(s)a^*(z(s))\leqslant0$ and $|h_N|=1$ on the boundary, we get
 \begin{align*}
\int_{\partial \Omega}|a^*(z)|\,|dz|&=- \int_{\partial \Omega}ia^*(z)\,dz=-i\,c^*\int_{\partial \Omega}\frac{f(z)a^*(z)}{h_N(z)}\,dz-\int_{\partial \Omega}\varphi (z)\,dz\\
&=\int_{\partial \Omega}\frac{c^*f(z)}{h_N(z)}|a^*(z)|\,|dz|+0\leqslant c^*\|f\|_{\infty}\int_{\partial \Omega}|a^*(z)|\,|dz|
 \end{align*}
 This implies $1\leqslant c^*\|f||_\infty$, with equality if $f=h_N/c^*$;  therefore,
 $m^*=\min \|f\|_\infty= 1/c^*$.
 Furthermore, if the function $f$ realizes equality, then
 \begin{align*}
\int_{\partial \Omega}|a^*(z)|\,|dz|=\int_{\partial \Omega}\frac{c^*f(z)}{h_N(z)}|a^*(z)|\,|dz|\quad\textrm{and } \Big|\frac{c^*f}{h_N}\Big|\leqslant1 \ \textrm{ a.e. on }\partial \Omega;
\end{align*}
this implies $\frac{c^*f}{h_N}=1$ a.e. on $\partial \Omega$ since the number of zeros of $a^*$ on the boundary is finite;  
thus, $c^*h_N =f$ in $ \Omega$, that proves uniqueness.
\end{proof}

\section{Interpolation with derivatives}\label{picder}
We now consider Hermite type interpolation. We still have $r$ distinct points $\zeta _j$, $j=1,\dots,r$, but now, to each $j$, we associate $s_j$ data 
$\tau _{j,0},\dots,\tau _{j,s_j-1}\in \C$\,; we call $\sigma_1= s_1{+}\cdots{+}s_k$ the number of these data. We look at the problem
\begin{align*}
(P_H)\quad\left\{\begin{matrix}
\textrm{Find a holomorphic function} f^* \textrm{ which realizes} ,\hfill\\[10pt]
m^*=\min\{\|f\|_\infty\,: f\ \in H^\infty(\Omega),\  f^{(s)}(\zeta _j)=\tau _{j,s},\  1\leqslant j\leqslant r, 0\leqslant s<s_j\}.
\end{matrix}
\right.
\end{align*}
We introduce the Hermite polynomial $\varpi$ of degree $\leqslant\sigma _1{-}1$ defined by
\[
\varpi^{(s)}(\zeta _j)=\tau _{j,s}, \quad 1\leqslant j\leqslant r, 0\leqslant s<s_j.
\]
Then the problem rewrites
 \begin{align*}
(P_H)\quad\left\{\begin{matrix}
\textrm{Find a holomorphic function} f^* \textrm{ which realizes} ,\hfill\\[10pt]
m^*=\min\{\|f\|_\infty\,: f\ \in H^\infty(\Omega),\  (f{-}\varpi)(\zeta)=O((\zeta {-}\zeta _j)^{s_j}),\  1\leqslant j\leqslant r\}.
\end{matrix}
\right.
\end{align*}
Such a function $f^*$ clearly exists\,; as previously, we also associate the problem
\begin{align*}
(P_{HN})\quad\left\{\begin{matrix}
\textrm{Find a function } h_N\in\EN \textrm{ which realizes} ,\hfill\\[10pt]
c_N=\max\{c\,: \exists \,h \in \EN,\  (h{-}c\varpi)(\zeta)=O((\zeta {-}\zeta _j)^{s_j}),\  1\leqslant j\leqslant r\}.
\end{matrix}
\right.
\end{align*}
It is easy to adapt the proof of Lemma\,\ref{lem4} to show that Problem $(P_{HN})$ has a solution as soon as $N\geqslant k{+}s_1{+}\cdots{+}s_r$ and that $c_N>0$ if  $N\geqslant 2k{+}s_1{+}\cdots{+}s_r$. 
We first  assume  that
$\tau_{j,0}\neq 0$ for all $j$; then, we may write Problem $(P_{HN})$ as

Find a function $h_N \in \EN$ which realizes
\begin{align*}
\gamma_N=\displaystyle\max_{\gamma \in\R,\,h\in\mathcal E_N}\{\gamma \,: \ \log( (h/ \varpi)(\zeta)){-}\gamma =O((\zeta {-}\zeta _j)^{s_j}),\   j=1,\dots,r\}.
\end{align*}
Note that the condition $\log( (h/ \varpi)(\zeta)){-}\gamma =O((\zeta {-}\zeta _j)^{s_j})$
is equivalent to
\[
\frac{d^s}{d\zeta ^s}(\log ((h/\varpi)(\zeta )){-}\gamma )\Big|_{\zeta =\zeta _j}=0,\quad  0\leqslant s<s_j.
\]
We choose such a solution $h_N\in\mathcal E_N$ of order $n\leqslant N$ that we write
\begin{align}\label{henne}
h_N(\zeta )=e^{i\theta^*}g_{\rho ^*}\varphi_{p _1^*}(\zeta )\dots\varphi_{p_N^*}(\zeta),\ \textrm{ avec  }\sum_{\ell=1}^N\omega_j(p^*_\ell)=-\rho _j^*, j=1,\dots,k.
\end{align}
$p_j^*\in \Omega$ if $1\leqslant j\leqslant n$, $p_j^*\in\partial \Omega$ if $j>n$. Furthermore, we assume that in the numbering, 
 equal $p_j^*$s are successive, we also suppose that the points $p_j^*\in \partial \Omega$, $j>n$, generate a dense set on the boundary
$\partial \Omega$.
We introduce $\e=(e_1,\dots,e_n)\in\C^n$, $\tv=(t_{n+1},\cdots,t_N)$  and we set
\begin{align*}
d_{\ell,s}(\gamma,\theta,\e,\tv)&:=\Re\frac{d^s}{d\zeta ^s}(\log ((h /\varpi)(\zeta )){-}\gamma )\Big|_{\zeta =\zeta _\ell},\\
d_{\ell+r,s}(\gamma,\theta,\e,\tv)&:=\Im\frac{d^s}{d\zeta ^s}(\log( (h /\varpi)(\zeta ))\Big|_{\zeta =\zeta _\ell},\quad\quad{ \ell=1,\dots,r,\ 0\leqslant s<s_j,}\\
d_{2\sigma _1+j}(\gamma,\theta, \e,\tv)&:=\sum_{\ell=1}^N\omega_j(p^*_\ell)+\rho _j^*, \quad j=1,\dots,k,
\end{align*}
where
\begin{align*}
h(\zeta )&=e^{i\theta }g_{\rho^*}(\zeta )\, \varphi_{p _1}(\zeta )\dots\varphi_{p_N}(\zeta),\\
p_j&=p_j^*{+}e_j \ \textrm { if }\ 1\leqslant j\leqslant n, \quad p_j=p_j^*{-}t_j\nu_j^* \ \textrm { if } j>n.
\end{align*}
There exists $\sigma >0$ such that, if $|e_j|<\sigma $ and if $0\leqslant t_j<\sigma $ it holds $p_j\in \overline{\Omega}$. Then $(\gamma _N,\theta_N,\mathbf 0,\mathbf 0)$ realizes
\begin{align*}
\gamma _N=\max\{\gamma\,:\  &d_{\ell,s}(\gamma,\theta,\e,\tv)=0,\ \textrm { if }\ 1\leqslant\ell\leqslant2r,0\leqslant s<r_j,\\&d_{2\sigma _1+j}(\gamma,\theta,\e,\tv)=0,\ \ \ \textrm { if }
1\leqslant j\leqslant k,\\ & \gamma,\theta  \in\R,\quad |e_j|<\sigma ,\ \textrm { if } 1\leqslant j\leqslant n,\quad 0\leqslant t_j<\sigma, \ \textrm { if } j>n\}.
\end{align*}
From the Fritz John optimality conditions, real numbers 
$\ell_0\ell_{1,0},{\dots}\ell_{2r,s_r-1}\ell_{2\sigma _1+1}{\dots}\ell_{2\sigma _1+k}$ and positive or null numbers $x_{n+1},\dots,x_N$, not all zero, exist such that
\begin{align*} 
\ell_{0}\nabla \gamma +\sum_{m=1}^{r}\sum_{s=0}^{s_j-1}(\ell_{m,s}\nabla d_{m,s}{+}&\ell_{m+r,s}\nabla d_{m+r,s})(\gamma _N,\theta_N,\mathbf 0,\mathbf 0)\\
&+\sum_{j=1}^k\ell_{2\sigma _1+j}\nabla d_{2\sigma _1+j}(\gamma _N,\theta_N,\mathbf 0,\mathbf 0)
+\sum_{j=n+1}^Nx_j\nabla t_j=0.
\end{align*}
This can also be written as
\[
\ell_0-\ell_{1,0}-\cdots-\ell_{r,0}=0, \quad \ell_{r+1,0}+\cdots+\ell_{2r,0}=0,
\]
\begin{align}
\sum_{m=1}^r\sum_{s=0}^{s_j-1}(\ell_{m,s}\nabla_j d_{m,s}&{+}\ell_{m+r,s}\nabla_j d_{m+r,s})(\gamma _N,\theta_N,\mathbf 0,\mathbf 0)\nonumber\\
&+\sum_{j=1}^k\ell_{2\sigma _1+j}\nabla_j d_{2\sigma_1+j}(\gamma _N,\theta_N,\mathbf 0,\mathbf 0)=0, \quad j=1,\dots,n, \label{lag2}
\end{align}
\begin{align}
\sum_{m=1}^r\sum_{s=0}^{s_j-1}(\ell_{m,s}\partial_{t_j}  d_{m,s}{+}\ell_{m+r,s}\partial_{t_j}  d_{m+r,s})(\gamma _N,\theta_N,\mathbf 0,\mathbf 0)&=-x_j, \quad j\geqslant n{+}1.
\label{john2}
\end{align}
Thus, the numbers $\ell_{1,0},\dots,\ell_{2r,s},\ell_{2\sigma _1+1},\dots,\ell_{2\sigma _1+k}$ are not all zero and it holds
\begin{align*}
\nabla_{\!j \,}d_{m,s}(\gamma _N,\theta_N,\mathbf 0,\mathbf 0)&=\nabla_{\!j \,}\Re\frac{d^s}{d\zeta ^s}\log\varphi _{p_j^*}(\zeta)\Big|_{\zeta  =\zeta _m}\nonumber\\
\nabla_{\!j \,}d_{m+r}(\gamma _N,\theta_N,\mathbf 0,\mathbf 0)&=\nabla_{\!j \,}\Im\frac{d^s}{d\zeta ^s}\log\varphi _{p_j^*}(\zeta)\Big|_{\zeta  =\zeta _m},\ m=1,\dots,r,\ 0\leqslant s<s_j,\\
\nabla _{\!j \,}d_{2\sigma _1+m} (\gamma _N,\theta_N,\mathbf 0,\mathbf 0)&=\nabla_{\!j \,}\omega_m(p_j^*),\ \ m=1,\dots,k.\nonumber
\end{align*}
We introduce the function
\begin{align*} 
H(z)=\sum_{m=1}^r\sum_{s=0}^{s_j-1}\big(\ell_{m,s}\,\Re\frac{d^s}{d\zeta ^s}\log\varphi _z(\zeta _m)|+\ell_{m+r,s}\frac{d^s}{d\zeta ^s}\Im\log\varphi_z(\zeta _m)\big)\\
+\sum_{m=1}^k\ell_{2\sigma _1+m}\omega_m(z).
\end{align*}
and we set, with $z=x{+}iy$,
\begin{align} \label{a2}
a(z)=\partial_x H(z)-i\,\partial_y H(z). 
\end{align}
Then, condition \eqref{lag2} reads $a(p_j^*)=0$, $j=1,\dots,n$.\medskip

As in Section\,\ref{pn}, it can been shown that the multipliers $\ell_m$ may be chosen independently 
of $N$, and the function $a$ is meromorphic in $\overline\Omega$ with roots $p_j^*$, $j=1,\dots,n$,
and possible poles $\zeta _j$ with order less than or equal to $s_j$, $1\leqslant j\leqslant r$.
Furthermore, $\Re(a(z(\sigma ))z'(\sigma ))=0$ on the boundary $\partial \Omega$ and it holds $n\leqslant\sigma_1{+}k{-}1$. 
From conditions \eqref{john2} related to the inequality constraints $t_j\geqslant 0$, we deduce
\begin{equation*}
i\,z'(\sigma )a(z(\sigma))\le 0,\quad\textrm{ on the boundary}.
\end{equation*}
In the case where the $p_j^*$ are not all distinct, the change of variables $\e\mapsto \uu$  allows to show that the multiplicity of  
$p_j^*$ as root of $h_N$  is less than or equal at its multiplicity as root of $a$. 

In order to lighten the notation, we have assumed up to now $\tau_{j,0}\neq 0$  for all $j$, but it is not difficult 
to treat the other cases. If $\zeta _j$ is a root of multiplicity $\nu_j$ of polynomial $\varpi$, it suffices in the representation
\eqref{henne} of $h_N$  to force $\nu_j$ $p_m^*$ to be equal to $\zeta _j$, as we have 
proceeded in Section\,\ref{pn}.\medskip

We have obtained the result
\begin{theorem} \label{thderiv}
 There exists a unique holomorphic function $f^*\in H^\infty(\Omega)$ which realizes
 \[
m^*=\min\{\|f\|_\infty\,: f\ \in H^\infty(\Omega),\  f^{(s)}(\zeta _j)=\tau _{j,s},\  1\leqslant j\leqslant k, 0\leqslant s<r_j\}.
\]
Furthermore, $h_N=f^*/m^*$ is an inner function of order $n\leqslant\sigma_1{+}k{-}1$ and it
is for all $N\geqslant n$ the unique solution of Problem $(P_{NH})$. 
\end{theorem}

Now, if $A\in \C^{d,d}$ is a $d\times d$ matrix with eigenvalues in $\Omega$, then,
it is easily seen that there exists a function $f_0\in H^\infty(\Omega)$ which realizes
 \[
 \|f_0(A)\|=\max\{\|f(A)\|\,: f\in H^\infty(\Omega), |f|\leqslant 1\textrm{ in  }\Omega\}.
 \]
Clearly, it holds $\|f_0\|_\infty=1$ and $1$ is the minimum norm among the functions $f\in H^\infty(\Omega)$ which satisfy $f(A)=f_0(A)$. In particular, $f_0$ realizes the minimum norm among the $f\in H^\infty(\Omega)$ such that
\[
f^{(s)}(\lambda )=f_0^{(s)}(\lambda ),\quad\textrm{for }0\leqslant s<r,
\]
for all eigenvalues of $\lambda $ with multiplicity $r$.  Hence, $f_0$ is a inner function of order $\leqslant d{+}k{-}1$, which shows Theorem\,\ref{th1}.

\section{Link with the original Pick-Nevanlinna problem}\label{link}

We turn back to Problem $(P)$ with distinct interpolation points $\zeta _1,\dots,\zeta _r$. We set
\begin{align}\label{m}
m^*(\zeta ,\tau )=\min\{\|f\|_\infty\,: f\ \in H^\infty(\Omega),\  f(\zeta _1)=\tau _1,\dots,f(\zeta _k)=\tau _r\}.
\end{align}
The solution $f^*$ of Problem $(P)$ is the unique holomorphic function
satisfying the interpolation conditions $f^*(\zeta _j)=\tau _j$, $1\leqslant j\leqslant r$ and $\|f^*\|=m^*(\zeta ,\tau )$; furthermore, $f^*/m^*(\zeta ,\tau )$ 
is an inner function of order $\leqslant r{+}k{-}1$. For $m>m^*(\zeta ,\tau )$,
there exist an infinite number of functions satisfying the interpolation conditions $f^*(\zeta _j)=\tau _j$ 
and such that $f/m$ be inner functions of order $\leqslant r{+}k$ (and thus $\|f\|_\infty=m$).
Indeed, let $\zeta _0\in\Omega$ be a point distinct of $\zeta _1,\dots,\zeta _k$ and choose $\tau _0>0\in \R$\,; we set $\tilde\zeta=(\zeta _0,\zeta _1,\dots,\zeta _k)$ and $\tilde\tau =(\tau_0,\tau_1,\dots,\tau_k)$.
It is clear that $m^*(\tilde\zeta ,\tilde\tau)\geqslant m^*(\zeta ,\tau )$, $m^*(\tilde\zeta ,\tilde\tau)\geqslant \tau_0$ and $m^*(\tilde\zeta ,\tilde\tau)=m^*(\zeta ,\tau )$ if $\tau _0=f^*(\zeta _0)$.
 We will see further that $m(\tilde\zeta ,\tilde\tau)$ depends continuously in $\tau _0$. This shows that $m^*(\tilde\zeta ,\tilde\tau)$ takes 
 all the values in the interval $[m(\zeta ,\tau ),\infty)$ as $\tau_0$ varies in $\R^+$.
In particular this shows that\,:

{\it The classical Pick-Nevanlinna problem has a solution if and only if $m^*(\zeta ,\tau )\leqslant 1$;
if, in addition, $m^*(\zeta ,\tau )=1$,  the solution is unique.}
\medskip

Note that $m^*(\zeta ,\lambda \tau )=|\lambda |\,m^*(\zeta ,\tau )$ for all $\lambda \in \C$. This function is locally Lipschitz with respect to the distinct points $\zeta _j\in\Omega$ and with respect to the $\tau _j\in \C$. More precisely, if for each $i=1,\dots,r$, we associate the polynomial: $\ell_{\zeta ,i}(x)=\prod_{j\neq i}\frac{x -\zeta _j}{\zeta_{j} -\zeta _i}$
and the function $\lambda _{\zeta ,i}$ which realizes (with the usual Kronecker symbol)
$\min\{\|f\|_\infty\,: f\ \in H^\infty(\Omega),\  f(\zeta _j )=\delta _{ij}, 1\leqslant j\leqslant r\}$, it is easily seen that\,:
\begin{align*} 
|m^*(\zeta ,\tau ){-}m^*(\zeta ,\tau' )|&\leqslant\sum_{i=1}^r\|\lambda _{\zeta ,i}\|_\infty\,|\tau _i{-}\tau '_i|,\\
|m^*(\zeta ,\tau ){-}m^*(\zeta' ,\tau)|&\leqslant\sum_{i=1}^r\| \ell_{\zeta ,i}{-}\ell_{\zeta' ,i}\|_\infty\,|\tau _i|.
\end{align*}
Note that $\|\lambda _{\zeta ,i}\|_\infty\leqslant\|\ell _{\zeta ,i}\|_\infty$.\medskip

Abrahamse has given another criterion for the existence of a solution to the Pick-Nevanlinna problem. Let us consider the $r\times r$ matrix \,: $M(m,\tau,\zeta,\alpha)$ with entries
\[
M_{i,j}(m,\tau ,\zeta ,\alpha)=(m^2-\tau _i\bar\tau _j)k^\alpha(\zeta _i,\zeta_j),
\]
where
$k^\alpha(\zeta_1,\zeta _2)$ is a family of kernels (of Poisson type) on $\Omega\times\Omega$ indexed by $\alpha\in \T^r$ (we refer to the paper \cite{abra} for their definitions).
The criterion of Abrahamse is : 

\textit{The Pick-Nevanlinna problem has a solution if and only if the matrix 
$M(1,\tau,\zeta,\alpha)$ is positive definite for all $\alpha\in \T^r$; this solution is unique if and only if 
furthermore there exists $\alpha\in \T^r$ such that {\rm det}$(M(1,\tau,\zeta,\alpha))=0$. }

This shows that $m^*(\zeta ,\tau )$ is the smallest value of $m>0$ such that $M(m,\tau ,\zeta,\alpha )$ is positive semidefinite for all $\alpha\in \T^r$.

In the simple case where $\Omega=\D$ is the unit disk, $k^\alpha(\zeta _i,\zeta_j)=1{-}\zeta _i \bar\zeta _j$, this is the famous Pick condition\,\cite{pick}. But, except in the case of the annulus, explicitness and computability of these kernels are not known.

\section{Appendix. Proof of Lemma\,\ref{lem7}. }

We consider $r$ given points $\zeta_1{,\dots},\zeta_r\in\Omega$. 
Let us note $e_1{=}(1,0,{\dots},0)^T$\!, $e_2{=}(0,1,0,{\dots},0)^T$\!,
\dots, $e_k=(0,\dots,0,1)^T$ the canonic basis of $\R^k$, and set $e=e_1{+}\cdots{+}e_k$.
We choose $k$ smooth paths $\{\gamma _j(\tau )\,: 0\leqslant\tau \leqslant1\}$ satisfying $\gamma _j(\tau )\in\Omega$ if $0<\tau <1$, $\gamma _j(0)\in \Gamma _0$, and $\Gamma _j(1)\in \Gamma _j$; 
we assume that these paths are disjoint and that they do not contain any of the given points $\zeta_1,\dots,\zeta_r\in\Omega$. 
We now consider the functions
\[
 \varphi _j(\tau )=(\omega_1(\gamma _j(\tau )),\omega_2(\gamma _j(\tau )),\dots,\omega_k(\gamma _j(\tau )))^T,\quad\textrm{defined for }\tau \in[0,1]
\]
that we extend as Lipschitz functions from $\R$ into $\R^k$ in such a way that
\[
\varphi _j(\tau {+}l)=\varphi _j(\tau )+l\,e_j, \quad\textrm{for all }l\in \Z.
\]
We now consider the functions (with $t=(t_1,\dots,t_k)^T\in\R^k$, $s\in [0,1]$)
\begin{align*}
\Psi_1(t)=\varphi_1(t_1){+}\cdots{+}\varphi_k(t_k)\quad\textrm{and}\quad \Psi _s(t)=(1{-}s)t+s\,\Psi _1(t).
\end{align*}
Note that $\Psi _s(t{+}le_j)=\Psi _s(t){+}le_j$ if $l\in\Z$, $j\in\{1,\dots,k\}$ and $\Psi _s(t)=t$ if $t\in\Z^k$.

\begin{lemma}\label{esp}
The image of $\R^k$ by $\Psi _s$ is the space $\R^k$ for all $0\leqslant s\leqslant 1$.
\end{lemma}
\begin{proof}
We consider the cube $C_k=[-k,2]^k$ and its boundary $\Sigma$; let
 \begin{align*}
\Sigma _{js}=\{\Psi_s (t)\,:t\in C_k, \ t_j=-k\},\quad
\Sigma  _{k+j,s}=\{\Psi_s (t)\,:t\in C_k, \ t_j=2\},
\end{align*}
be the images by $\Psi _s$ of the faces of $C_k$ and $\Sigma _s=\Psi _s(\Sigma )=\cup_{j=1}^{2k}\Sigma _{js}$. 
Notice that the intersection of $\Sigma _s$ with the unit cube $C=[0,1]^k$ is empty. Indeed, if, for instance, 
$x\in\Sigma _{1s}$,  $x=\Psi _s(-k,t_2,\dots,t_k)$, then $\big(\varphi _1(-k)\big)_1=-k$ while $\big(\varphi _j(t_j)\big)_1=\omega_1(\gamma _j(t_j{-}[t_j])\in[0,1]$, $j=2,\dots,k$; this implies $x_1=-k{+}s\sum_{j=2}^k\big(\varphi _j(t_j)\big)_1\leqslant-1$
 (we have denoted by $[t_j]$ the integer part of $t_j$).
 Similarly, if $x=\Psi _s(2,t_2,\dots,t_k)\in\Sigma _{k+1,s}$, then $\big(\varphi _1(1)\big)_1=2$ and we get $x_1=2{+}s\sum_{j=2}^k\big(\varphi _j(t_j)\big)_1\geqslant2$.
\medskip

We will denote by $a(t,S)$ the solid angle of a surface $S\subset \R^k$ observed from a point $t\notin S$
and $\varpi_k$ the constant which corresponds to the solid angle of the unit sphere observed from its
center.
 Let us consider a point $d$ belonging to the unit cube; then,
 $a(d,\Sigma)=\varpi_k$ and we have seen that $\forall s\in [0,1]$, $d\notin\Sigma _s$. Thus, the function
 $a(d,\Sigma _s)$ is continuous with respect to $s$ with values in $\varpi_k \Z$, whence it is constant
  equal to $\varpi_k $.

We now show that the compact set $\Psi (C_k)$ contains the unit cube $C$. Indeed, otherwise there 
would exist a point $d\in C$ 
and $\varepsilon >0$ such that the ball $\{\|x{-}d\|\leqslant \varepsilon\} $ does
not meet $\Psi (C_k)$.
As $\theta $ varies from $0$ to $1$, the surface $\Sigma _s(\theta )=\Psi_s((1{-}\theta )\Sigma)$ is continuously distorted from $\Sigma _s(0)=\Sigma _s$ towards the point $\Sigma _s(1)=\Psi _s(0)=0$; furthermore, there holds $\Sigma _s(\theta )\subset\Psi(C_k)$.
The solid angle $a(d,\Sigma _s(\theta ))$ is a continuous function of $\theta $ with values in $\varpi_k \Z$, thus a constant function. In fact, it clearly holds $a(d,\Sigma _s(1))=0$ since $\Sigma _s(1)$ is reduced to a point;
therefore $a(d,\Sigma _s)=a(d,\Sigma _s(0))=0$. This shows that, if $d$ belongs to the unit cube, then $d\in \Psi _s(\R^k)$,  $\forall s\in [0,1]$. 
The image $\Psi_s (\R^k)$ being invariant by $\Z^k$ translations, we deduce that $\Psi _s(\R^k)=\R^k$.
  \end{proof}

\begin{proof}[Proof of Lemma\,\ref{lem7}] 
Let $\zeta _1,\dots,\zeta _r\in\Omega$ be $r$ given points; we set $d_j=-\sum_{l=1}^r\omega_j(\zeta _l)$. 
There exists $t\in \R^k$ such that $d=\Psi _1(t)$. We denote by $[t]$ the vector of $\R^k$ with components $[t_j]$. 
By setting $z_{j+r}=\gamma _j(t_j{-}[t_j])\in\overline\Omega$, we deduce from $\Psi _1(t{-}[t])=d-[t]$ that, for $j=1,\dots,k$,
\[
0<\omega_j(\zeta _1){+}\cdots{+}\omega_j(\zeta _{r+k})=\omega_j(\zeta _{r+1}){+}\cdots{+}\omega_j(\zeta _{r+k}){-}d_j=-[t_j]\in\N.
\]
We can suppress $\zeta _{r+l}$ of these summation if $\zeta _{r+l}\in \partial \Omega$ and still have
$\omega_j(\zeta _1){+}\omega_j(\zeta _{r+k'})\in\N$ with all $\zeta _j\in\Omega$ (eventually $k'=0$
if $d\in\Z^k$). The choice of the paths $\gamma _j$ allows to have the new points distinct and distinct from $\zeta _1,\dots\zeta _r$,
\end{proof}

\end{document}